\documentclass[12pt, reqno]{amsart}

\usepackage[english]{babel}
\usepackage[utf8x]{inputenc}
\usepackage[T1]{fontenc}
\usepackage[all,cmtip]{xy}

\usepackage[alphabetic, initials, lite]{amsrefs} 
\usepackage{fullpage}
\usepackage{setspace}
\usepackage{amssymb,amsmath,amsthm,amscd,epsfig}
\usepackage{mathtools}
\usepackage{graphicx}
\usepackage[colorinlistoftodos]{todonotes}
\usepackage[colorlinks=true, allcolors=blue]{hyperref}

\newcommand{\hiro}[1]{{\color{blue} \sf $\clubsuit\clubsuit\clubsuit$ Hiro: [#1]}}

\newtheorem{thm}{Theorem}[section]
\newtheorem{cor}[thm]{Corollary}
\newtheorem{prop}[thm]{Proposition}

\theoremstyle{definition}

\newtheorem{example}[thm]{Example}
\newtheorem{remark}[thm]{Remark}

\newtheorem{Lemma}[thm]{Lemma}

\theoremstyle{remark}

\numberwithin{equation}{section}

\def\Z{\ifmmode{{\mathbb Z}}\else{${\mathbb Z}$}\fi}
\def\Q{\ifmmode{{\mathbb Q}}\else{${\mathbb Q}$}\fi}
\def\C{\ifmmode{{\mathbb C}}\else{${\mathbb C}$}\fi}
\def\P{\ifmmode{{\mathbb P}}\else{${\mathbb P}$}\fi}
\def\H{\ifmmode{{\mathrm H}}\else{${\mathrm H}$}\fi}
\def\G{\ifmmode{{\mathbb G}}\else{${\mathbb G}$}\fi}
\def\R{\ifmmode{{\mathbb R}}\else{${\mathbb R}$}\fi}
\def\F{\ifmmode{{\mathbb F}}\else{${\mathbb F}$}\fi}
\def\O{\ifmmode{{\cal O}}\else{${\cal O}$}\fi}
\def\D{\ifmmode{{\cal{D}}^b}\else{${{\cal{D}}^b}$}\fi}

\newcommand{\calA}{{\mathcal A}}
\newcommand{\calB}{{\mathcal B}}
\newcommand{\calC}{{\mathcal C}}

\newcommand{\calO}{{\mathcal O}}

\DeclareMathOperator{\inv}{inv}

\DeclareMathOperator{\Gal}{Gal}

\DeclareMathOperator{\ev}{ev}

\DeclareMathOperator{\et}{et}

\DeclareMathOperator{\N}{N}

\DeclareMathOperator{\Br}{Br}

\usepackage{todonotes}

\newcommand{\A}{\mathbb{A}}

\newcommand{\isom}{\simeq}

\usepackage[shortlabels]{enumitem}
\usepackage{tikz-cd}

\title{Weak approximation on Ch\^{a}telet surfaces}

\author{Masahiro Nakahara}
\author{Samuel Roven}


\begin{document}

\begin{abstract}
We study weak approximation for Ch\^{a}telet surfaces over number fields when all singular fibers are defined over rational points.  We consider Ch\^{a}telet surfaces which satisfy weak approximation over every finite extension of the ground field. We prove many of these results by showing that the Brauer-Manin obstruction vanishes, then apply results of Colliot-Th\'el\`ene, Sansuc, and Swinnerton-Dyer. 
\end{abstract}

\maketitle

\section{Introduction}

A Ch\^{a}telet surface over a number field $k$ is a smooth projective model of the affine surface given by the equation
\begin{equation}\label{eqn:Chatelet}
y^2-az^2=P(x)
\end{equation}
where $a\in k^\times\setminus k^{\times2}$ and $P(x)$ is a separable polynomial of degree $3$ or $4$.  The arithmetic of these surfaces have been studied extensively by Colliot-Th\'el\`ene, Sansuc, and Swinnerton-Dyer in \cites{CTSSD87,CTSSD87b}, where it was proven that over a number field, rational points on these surfaces are controlled by the Brauer--Manin obstruction. More precisely, if $X$ is a Ch\^{a}telet surface over a number field $k$, then $X$ may fail weak approximation, i.e., $X(k)$ is not dense in $X(\A)$, but it is always dense in the Brauer--Manin set $X(\A)^{\Br}$.

In this paper, we focus on weak approximation on Ch\^{a}telet surfaces. Since the Brauer group is largest when $P(x)$ splits into linear factors (see \S2), one might expect that weak approximation fails. In this case, we obtain a complete characterization for when $X$ fails weak approximation.

\begin{thm}\label{thm:SurjLinear}
Let $X$ be a Ch\^{a}telet surface over a number field $k$ with $X(k)\neq\emptyset$. Assume that $P(x)$ splits into linear factors. Then $X$ fails weak approximation if and only if either
\begin{enumerate}
    \item $X$ has a place $v$ of bad reduction where $a\notin k_v^{\times2}$ or,
    \item $k$ has a real embedding where $a<0$.
\end{enumerate}
\end{thm}

See \S\ref{sec:split} for what we mean by bad reduction. Roughly, this translates to either $k_v(\sqrt{a})k_v$ being ramified or $P(x)$ being nonseparable modulo $v$.

These conditions suggest that we expect weak approximation to fail in general. This is in contrast to the case where $P(x)$ is irreducible or is a product of a linear and irreducible cubic where $X$ always satisfies weak approximation \cite{CTSSD87}*{Theorem B}. We also collect some partial results and show that the situation is not as simple in the remaining case when $P(x)$ is a product of two irreducible quadratics, see Examples \ref{ex:WAsatisfied}, \ref{ex:WAfailed}.

For Ch\^atelet surfaces, behavior of the Hasse principle or weak approximation under finite extensions have been studied in \cites{Lia18,Wu22,Rov22}. In \cites{Lia18,Wu22}, the authors give a Ch\^atelet surface over any arbitrary number field $k$, which satisfies weak approximation over $k$ but fails over some finite extension. On the other hand, there exist Ch\^atelet surfaces that satisfy weak approximation over every finite extension $K/k$ (e.g. when $X$ is rational). Using Theorem \ref{thm:SurjLinear}, we give a list of conditions for Ch\^{a}telet surfaces to fall into either category.

\begin{thm}\label{thm:BrauerVanishIntro}
Let $X$ be a Ch\^{a}telet surface over a number field $k$. Let $L$ be the splitting field of $P(x)$.
\begin{enumerate}
    \item $X$ fails weak approximation over some finite field extension if either (1) or (2) in Theorem \ref{thm:SurjLinear} is satisfied for $X_L$.
    \item $X$ satisfies weak approximation over all finite field extensions if any of the following conditions are satisfied, (Here $D_{2n}$ denotes the dihedral group of order $2n$)
    \begin{enumerate}
        \item if both (1) and (2) of Theorem \ref{thm:SurjLinear} fail for $X_L$ and $\Gal(L/k)\isom \Z/3\Z,\{1\}$,
        \item if $\Gal(L/k)\isom D_8,(\Z/2\Z)^2$ and $P(x)$ factors further over $k(\sqrt{a})/k$,
        \item if $\Gal(L/k)\isom D_6,\Z/4\Z,\Z/2\Z$ and $\sqrt{a}\in L$.
    \end{enumerate}
\end{enumerate}
\end{thm}

We prove (1) and (2)(a) as a direct corollary of Theorem \ref{thm:SurjLinear} in \S4 and the proof for (2)(b)(c) is given in \S2 where we discuss the Brauer group of Ch\^atelet surfaces.

For (2)(b)(c), the conditions in fact imply that $\Br X_K/\Br_0 X_K=0$ for all finite extensions $K/k$. This is enough to conclude that $X_K$ satisfies weak approximation. In fact, when $\Gal(L/k)\isom D_6$ and $L$ contains $\sqrt{a}$, then $X$ is stably rational \cite{BCTSSD}*{Theorem 1}, from which it follows in particular that $X$ satisfies perpetual weak approximation. In the instances when $\sqrt{a}$ is contained in the residue field of one of the factors of $P(x)$, $X$ is rational.

In (2)(a), $\Br X_L/\Br_0 X_L\isom(\Z/2\Z)^2$ so in particular $X$ is not rational. See Example \ref{ex:perp} for a surface in this category.

In all the instances not covered by Theorem \ref{thm:BrauerVanishIntro}, there is a subextension $K/k$ inside $L$ such that $P(x)$ factors into a product of two irreducible quadratics over $K$ and $\sqrt{a}\in L$. Then $X$ satisfies perpetual weak approximation if and only if $X_K$ satisfies weak approximation. 


\subsection{Acknowledgments}
We are grateful to Bianca Viray for helpful discussions and comments. We also thank Daniel Loughran, Jean-Louis Colliot-Th\'el\`ene for comments on the initial draft of the paper.

\section{Cyclic algebras over local fields}

In this section let $k$ be a nonarchimedian local field containing an $n$th roots of unity $\zeta$. Let $\calO$ be its ring of integers with uniformizer $\pi\in\calO$ and $\F_q=\calO/(\pi)$. Recall that the invariant map gives an isomorphism $\inv\colon \Br k\to \Q/\Z$.

For $a,b\in k^\times$, let $(a,b)_\zeta=(a,b)_{\zeta}\in \Br k$ be the class of the cyclic algebra as defined in \cite{GS17}*{\S2.5}. When $n=2$, we will simply write $(a,b)=(a,b)_{-1}$.

\begin{Lemma}\label{lem:HilbUnram}
If $k(\sqrt[n]{a})/k$ is unramified, then for any $b\in k^{\times}$, $\inv (a,b)_\zeta=v_\pi(b)/n$.
\end{Lemma}

\begin{proof}
\cite{Ser86}*{\S2.5 Proposition 2}
\end{proof}

\begin{Lemma}\label{lem:nontrivialsymbol}
Assume the characterisitic of $\F_q$ is even. If $k(\sqrt{a})/k$ is ramified, then there exists $u\in\calO$ such that $\inv (a, 1-\pi u)=1/2$.
\end{Lemma}

\begin{proof}
By local class field theory, there exists some $b\in \calO^\times$ such that $b\notin \N_{k(\sqrt{a})/k}k(\sqrt{a})^\times$ and so $\inv (a,b)=1/2$. Since the squaring map $x\mapsto x^2$ is an isomorphism on $\F_q^{\times}$, there exists $d\in \calO^\times$ such that $b-d^2\equiv 0\bmod \pi$. Write $b-d^2=-\pi e$ for some $e\in\calO$. Then,
$$(a,b)=(a,d^2+b-d^2)=(a,d^2-\pi e)=(a,1-\pi(e/d^2)).$$
Setting $u=e/d^2\in\calO$ gives our result.
\end{proof}

\section{Generators for the Brauer group}\label{sec:Brauergroup}

Throughout the rest of this paper, $X$ will denote a Ch\^atelet surface over a field $k$ of characteristic $\neq2$ with affine model given by the equation
$$y^2-az^2=P(x)$$
where $a\in k^\times\setminus k^{\times 2}$ and $P(x)$ is a separable polynomial of degree $3$ or $4$. The morphism $X\to \P^1$ given by the $x$-coordinate gives $X$ the structure of a conic bundle. The singular fibers are precisely the roots of $P(x)$ together with $\infty\in\P^1$ if $\textup{deg}(P(x))=3$.

Let $\Br X=\H^2_{\et}(X,\G_m)$ denote the cohomological Brauer group of $X$. Let $\Br_0 X$ be the image of the natural map $\Br k\to \Br X$. The Brauer group of Ch\^{a}telet surfaces, or more generally conic bundles, have been extensively studied, see \cite{CTS21}*{\S11.3} for a detailed exposition. We summarize the necessary results for our purposes here.

\begin{prop}\label{Brauergroupgenerators}
The Brauer group of $X$ depends on the factorization of $P(x)$ over $k$. Then $\Br(X)/\Br_0(X)$ is isomorphic to
\begin{enumerate}
    \item $(\Z/2\Z)^2$ if $P(x)$ splits completely into linear factors,
    \item $\Z/2\Z$ if $P(x)$ has an irreducible quadratic factor and $\sqrt{a}\notin k(\alpha)$ for all roots $\alpha$ of $P(x)$,
    \item $\{0\}$ otherwise.
\end{enumerate}
\end{prop}

The points on $\P^1_k$ corresponding to the singular fibers determine the generators for the quotient $\Br X/\Br_0 X$. In particular, if $P(x)$ splits completely, we can map three of the roots of $P(x)$ to $0,1,$ and $\infty \in \P^1$ via an automorphism of $\P^1$. In doing so, we may assume that $P(x)=cx(x-1)(x-\lambda)$ for some $\lambda \in k \smallsetminus \{0,1\}$ and $c\in k$.

The main proofs of this paper work with fixed generators for the Brauer groups, and for the cases above where $\Br X/\Br_0 X\neq 0$, our generating algebras are as follows:
\begin{enumerate}
\item If $P(x)$ splits completely, then $\Br X/\Br_0 X$ is generated by quaternion algebras of the form $$\mathcal{A}=(a,x(x-1)) \;\;\; \text{and} \;\;\; \mathcal{B}=(a, x(x-\lambda)).$$
\item If $P(x)=cf(x)g(x)$ where both $f$ and $g$ are monic and at least one of $f(x)$ or $g(x)$ of degree $2$, then $\Br X/\Br_0 X$ is generated by the quaternion algebra $$\mathcal{C}=(a,f(x))=(a,cg(x)).$$ 

\end{enumerate}

We now prove Theorem \ref{thm:BrauerVanishIntro}(2)(b)(c) from the introduction. It suffices to show that the Brauer group consists of only constant algebras, in which case there is no Brauer--Manin obstruction to weak approximation.

\begin{thm}\label{thm:BrauerVanish}
Let $X$ be a Ch\^{a}telet surface over $k$ and $L/k$ the splitting field of $P(x)$. Then $\Br X_K/\Br_0 X_K=0$ for every finite extension $K/k$ if and only if one of the following holds
\begin{enumerate}
    \item $\Gal(L/k)\isom D_8,(\Z/2\Z)^2$ and $P(x)$ factors further over $k(\sqrt{a})/k$.
    \item $\Gal(L/k)\isom D_6,\Z/4\Z,\Z/2\Z$ and $\sqrt{a}\in L$.
\end{enumerate}
\end{thm}

\begin{proof}
Assume that either condition (1) or (2) listed above holds. Let $K/k$ be a finite extension. By Proposition \ref{Brauergroupgenerators}, we only need to consider the case when $P(x)$ splits into quadratic or smaller factors over $K$. Since $X_K$ is rational if $\sqrt{a}\in K$, we also assume $\sqrt{a}\notin K$. In particular this means $L\not\subseteq K$. We break into cases based on its factorization.

\begin{enumerate}
    \item (Product of two quadratics) In this case $\Gal(L/K)$ is either $\Z/2\Z$ or $(\Z/2\Z)^2$. In the former case, $L=K(\sqrt{a})$ and $\sqrt{a}$ is in the residue field of one of the quadratic factors, so $\Br X_K/\Br_0 X_K=0$ by Proposition \ref{Brauergroupgenerators}. In the latter case, we must have $\Gal(L/k)$ either $D_8$ or $(\Z/2\Z)^2$. By assumption, over $K(\sqrt{a})$, $P(x)$ must contain a linear factor. This implies that over $K$, one of the quadratic factors on $P(x)$ has residue field $K(\sqrt{a})$. 
    
    \item (Product of two linear and one quadratic) In this case $L=K(\sqrt{a})$ and so the quadratic factor of $P(x)$ has residue field containing $\sqrt{a}$.
\end{enumerate}

Conversely assume $\Br X_K/\Br_0 X_K=0$ for every finite extension $K/k$. If $\sqrt{a}\notin L$ then $\Br X_L/\Br_0 X_L\isom(\Z/2\Z)^2$ by Proposition \ref{Brauergroupgenerators}, so we may assume that $\sqrt{a}\in L$. We eliminate some possibilities by dividing into cases based on the isomorphism class of $\Gal(L/k)$.
\begin{itemize}
\item[$(S_4)$] There is a unique quadratic extension contained in $L/k$ corresponding to $A_4\subset S_4$ which must be $k(\sqrt{a})$. Let $K/k$ be an extension such that $\Gal(L/K)\isom(\Z/2\Z)^2$. Then $P(x)$ factors over $K$ as a product of two quadratics. Since $\Gal(L/K(\sqrt{a})$ must move both roots, it follows the residue fields of either factors does not contain $\sqrt{a}$. Hence $\Br X_K/\Br_0 X_K\isom\Z/2\Z$.

\item[$(A_4)$] There are no index $2$ subgroups in this case.

\item[$(D_8)$] If $P(x)$ remains irreducible over $k(\sqrt{a})$, then let $K/k$ be any distinct quadratic extension so that $P(x)$ factors into two quadratics over $K$. It follows neither of the residue fields contain $\sqrt{a}$ by the initial assumption so that $\Br X_K/\Br_0 X_K\isom\Z/2\Z$.

\item[$(K_4)$] If $P(x)$ is irreducible over $k$ then it must factor over $k(\sqrt{a})$. Otherwise, assume $P(x)$ is a product of two quadratics and remains so over $k(\sqrt{a})$. Then the same conclusion as the case above applies.
\end{itemize}

\noindent This leaves us with either (1) or (2) highlighted in the theorem.\end{proof}

\begin{cor}
Let $P(x)$ satisfy either conditions (1) or (2) of Theorem \ref{thm:BrauerVanish}. Then the Ch\^{a}telet surface
$$y^2-az^2=P(x)$$
satisfies weak approximation over all finite extensions $K/k$.
\end{cor}

\begin{proof}
By Theorem \ref{thm:BrauerVanish}, there is no Brauer--Manin obstruction to weak approximation over all such extensions $K/k$. So by \cite{CTSSD87}*{Theorem B(ii)(a)} they satisfy weak approximation.
\end{proof}

\section{Failure of weak approximation when \texorpdfstring{$P(x)$}{} is split}\label{sec:split}

Let $k$ be a number field. In this section, we consider the case when $P(x)$ factors into linear factors. Our goal is to prove Theorem \ref{thm:SurjLinear}. As discussed in \S\ref{sec:Brauergroup}, after moving one of the singular fibers to $\infty\in\P^1$, we may assume
\begin{equation}\label{eqn:linearfactors}
P(x)=cx(x-1)(x-\lambda)
\end{equation}
for some $\lambda\in k$ not equal to $0,1$ and $c\in k^\times$.

Let $v$ be a finite place of $k$ where $a\notin k_v^{\times2}$, and $\pi\in \calO_k$ be a uniformizer. After a change of coordinates, we may assume that $v(c)=0,1$. By considering the six cross-ratios on $\P^1$, we may also assume that $v(\lambda)\geq 0$. This allows us to define a model over $\calO_v$ using the same equation. We say that $X$ has bad reduction at $v$ if the special fiber of this model is singular. Concretely, this means that at least one of $v(c),v(\lambda),v(\lambda-1)$ is nonzero or that $k(\sqrt{a})/k$ is ramified at $v$.

Suppose the Ch\^{a}telet surface $X$ has bad reduction at $v$. We analyze the surjectivity of the evaluation map $\ev_{\calA}\colon X(k_v)\to \Br k_v[2]$ for the algebra $\calA\in\Br(X)$ listed in \S\ref{sec:Brauergroup}. We first consider the unramified case, and then deal with the ramified case separately for odd and even primes.

\subsection{Unramified Case} Assume that $k_v(\sqrt{a})/k_v$ is unramified. For $x\in k_v$, there exists a point in $X(k_v)$ lying over the corresponding point in $\P^1$ if $v(P(x))$ is even. By Lemma \ref{lem:HilbUnram}, for any point $Q=(x,y,z)\in X(k_v)$,
$$\inv_v(\calA(Q))=\inv_v(a,x(x-1))=v(x(x-1))/2\in \Q/\Z.$$
The map $\ev_\calA\colon X(k_v)\to \Br k_v[2]\isom\Z/2\Z$ is surjective if and only if it is nonconstant, and for $\calA(Q)$ to be nontrivial it is equivalent to $v(x(x-1))$ being odd.

\begin{prop}\label{prop:linearunramA}
If $k_v(\sqrt{a})/k_v$ is unramified, then the evaluation map $\ev_\calA\colon X(k_v)\to \Br k_v[2]$ is surjective.
\end{prop}

\begin{proof}
The assumptions imply we must have $v(c)\neq0$ so that $v(c)=1$ by our initial setup. Observe that for $x\in k_v$, there exists $Q\in X(k_v)$ lying over $x$ if and only if $v(x(x-1)(x-\lambda))$ is even. So it suffices to prove there exists $x_1,x_2\in k_v$ such that $P(x_1),P(x_2)$ are even and $x_1(x_1-1)\not\equiv x_2(x_2-1)\bmod 2$.

First assume that $v(\lambda)=v(\lambda-1)=0$. Take $x=1/\pi$, then $v(cP(x))=-2$ and $v((1/\pi)(1/\pi-1))=-2$ is even. On the other hand, if we take $x=\pi$, then $v(cP(x))=2$ and $v((\pi)(\pi-1))=1$ is odd.

To finish, it suffices to consider when $v(\lambda)>0$ (If $v(\lambda-1)>0$, then we may reduce to this case after a change of coordinates). We divide further into two cases, namely $v(c)=0$ and $v(c)=1$. Recall that $$v(P(x))=v(c)+v(x)+v(x-1)+v(x-\lambda).$$

If $v(c)=0$ then first set $x=\frac{1}{\pi}$ so that $v(P(x))=-6$ and $v(x(x-1))=-4$. Next, pick $x$ such that $0<v(x)$ and $v(x)\equiv v(x-\lambda)\equiv  1 \bmod 2$. Then $v(P(x))$ is even and $v(x(x-1))=v(x)$ is odd.

If $v(c)=1$, then first set $x=\frac{1}{\pi}$ so that $v(P(x))=-2$ and $v(x(x-1))=-2$. Furthermore, if we set $x=1+\pi$, then $v(P(x))=2$ and $v(x(x-1))=1$ as desired.
\end{proof}

\subsection{Ramified case odd}

If $v$ is an odd place and $k_v(\sqrt{a})/k_v$ is ramified, then we may assume after a change of coordinates that $v(a)=1$.

\begin{prop}\label{prop:linearoddram}
Assume that $v(a)=1$ and $P(x)$ has the form \eqref{eqn:linearfactors}. Then the evaluation map $\ev_\alpha\colon X(k_v)\to \Br k_v[2]$ is surjective for some $\alpha\in\{\calA,\calB\}$.
\end{prop}

\begin{proof}
Dividing the equation \eqref{eqn:Chatelet} for $X$ by a power of $-a$ and exchanging $y,z$ if necessary, we can also assume that $v(c)=0$. Let $R\in X(k)$ be the unique rational point over $\infty\in\P^1$. By using the following alternative formula for $\calA,\calB$ defined at $R$,
$$\calA=(a,w(1-w)),\calB=(a,w(1-\lambda w)),$$
where $w=1/x$, we see that $\ev_\calA(R)=\ev_\calB(R)=0$. To finish the proof, it suffices to exhibit a point where the evaluation is nontrivial.

Assume first that $\lambda\not\equiv0,1\bmod \pi$. Let $E\subset X$ be the closed subset given by $z=0$. Then $E$ is an elliptic curve given by
$$y^2=cx(x-1)(x-\lambda).$$
For $Q=(x,y)\in E(k_v)$, the quantities
$$cx,\ cx-c,\ cx-c\lambda$$
are all squares in $k_v$ if and only if $Q\in 2E(k_v)$ (use \cite{Hus04}*{\S1.4 Theorem 4.1} after suitable coordinate change).

\begin{Lemma}
$E(k_v)\setminus 2E(k_v)$ is nonempty
\end{Lemma}

\begin{proof}
Since $v\nmid 2$, $E$ has good reduction at $v$ and $E(k_v)[2^\infty]=\widetilde{E}(\F_v)[2^\infty]$, where $\widetilde{E}$ is the reduction modulo $v$. In particular, this means $E(k_v)[2^\infty]$ is finite and moreover nontrivial as $E(k_v)[2]\isom (\Z/2\Z)^2$. It follows then that $E(k_v)\setminus 2E(k_v)$ is nonempty.\end{proof}

\begin{remark}
Although we will not need it, the previous lemma is true for $v\mid 2$ as well. Then there is a subgroup $E^*\subset E(k_v)$ of finite index such that $\calO_v\isom E^*$. Take any $u\in \calO_v$ with $v(u)=0$. Then $u$ is not divisible by $2$ in $E^*$, and iteratively dividing by $2$ in $E(k_v)$ would produce infinitely many points in the quotient $E(k_v)/E^*$, a contradiction. Hence $u$ is not $2$-divisible in $E(k_v)$.
\end{remark}

Let $Q\in E(k_v)\setminus 2E(k_v)$ so then exactly two of $cx, cx-c, cx-c\lambda$ are not squares. This means in particular that $v(x)\geq0$, and so at most one of $v(x),v(x-1),v(x-\lambda)$ can be nonzero and they must all be even since $v(x(x-1)(x-\lambda))$ is even. We also obtain that at least one of $x(x-1),x(x-\lambda)$ is not a square. Hence, at least on of those products is not a square with even valuation which implies it is not a norm from $k_v(\sqrt{a})/k_v$. Hence, $\inv_v \alpha(Q)=1/2$ for some $\alpha\in\{\calA,\calB\}$.

Now assume that $\lambda\equiv 0\bmod \pi$. Choose $\overline{x}\in\F_v^\times\setminus \F_v^{\times2}$ such that $\overline{x}-1\in\F_v^{\times2}$. Lift $\overline{x}$ to $x\in\calO_v$. Then $(a,x)=(a,x-\lambda)\neq0$ but $(a,x-1)=0$. Hence, there exists some point $Q\in X(k_v)$ lying over $x\in\P^1$ where
$$\inv_v\calA(Q)=(a,x(x-1))=1/2.$$
The case $\lambda\equiv 1\bmod \pi$ is very similar.
\end{proof}

\subsection{Ramified case even}

\begin{prop}\label{prop:linearevenram}
Let $v$ be a place lying over $2$. Assume that $k_v(\sqrt{a})/k_v$ is ramified and $P(x)$ has the form \eqref{eqn:linearfactors}. Then the evaluation map $\ev_\calA\colon X(k_v)\to \Br k_v[2]$ is surjective.
\end{prop}

We give the proof of this result after establishing some basic facts on the distribution of norms inside $\calO_v$. For the remainder of this section, $v$ will denote a place lying over $2$ where $k_v(\sqrt{a})/k_v$ is ramified. Let $w$ be the place lying over $v$ and $L_w\coloneqq k_v(\sqrt{a})$. Let $\N\colon L_w\to k_v$ denote the norm map.

\subsubsection{Equidistribution of norms among residues}
The subgroup of norms $\{x\in\calO_v^{\times}\mid x\in \N(\calO_w^\times)\}$ has index $2$ inside $\calO_v^\times$.  For any subset $H\subset k_v$, let $H\bmod \pi^n$ denote the set of equivalence classes $H/\sim$ where $h_1\sim h_2$ if $h_1-h_2\in \pi^n\calO_v$.

\begin{Lemma}\label{lem:normdist}
Let $r\in \calO_v$. Then
$$\lim_{n\to\infty} \frac{\#\{x\in\calO_v\mid x\in \N(\calO_w),x\equiv r\bmod \pi\}\bmod \pi^n}{\#\{x\in\calO_v\mid x\equiv r\bmod \pi\}\bmod \pi^n}=\frac{1}{2}.$$
\end{Lemma}

\begin{proof}
Let $\calO^{(r)}_v\coloneqq \{x\in \calO_v \mid x\equiv r\bmod \pi\}$. We first prove the case when $r=1$. Note that $\N(\calO_w)\cap\calO_v^{(1)}\subset \calO_v^{(1)}$ is a subgroup of index at most $2$ under the multiplicative structure. By Lemma \ref{lem:nontrivialsymbol}, there exists $u\in\calO_v$ such that $1-\pi u\notin \N(\calO_w)$, and so $1-\pi u\in \calO_v^{(1)}\setminus (\N(\calO_w)\cap\calO_v^{(1)})$. Hence it follows $\N(\calO_w)\cap\calO_v^{(1)}$ has index $2$. Consider the quotient map
$$q_n\colon \calO_v^{(1)}\to \calO_v^{(1)}/\pi^n.$$
The image $q_n(\N(\calO_w)\cap\calO_v^{(1)})$ has either index $1$ or $2$. Since norms of finite index are open in the $v$-adic topology, it follows that for $n$ large enough, this image has index $2$. The statement about the limit follows immediately.

Now assume $r\not\equiv0\bmod \pi$. Take any $x\in\calO_v$ such that $x\equiv r\bmod \pi$. Multiplication by $1/x$ gives a bijection $\calO_v^{(r)}\to \calO_v^{(1)}$. Depending on $x$, this map sends $\N(\calO_w)\cap\calO_v^{(r)}$ to either $\N(\calO_w)\cap\calO_v^{(1)}$ or $\calO_v^{(1)}\setminus \N(\calO_w)\cap\calO_v^{(1)}$. Hence, the limit then follows from what we proved for $\calO_v^{(1)}$ above.

Finally assume $r=0$. Then noting that $-\pi\in \N(\calO_w)$ and setting $x'=x/(-\pi)$ gives
$$\{x\in\calO_v\mid x\in \N(\calO_w),x\equiv0\bmod \pi\}=-\pi\{x'\in\calO_v\mid x'\in \N(\calO_w)\}.$$
Hence the limit in question is
$$\lim_{n\to\infty} \frac{\#\{x'\in\calO_v\mid x'\in \N(\calO_w)\}\bmod \pi^{n-1}}{\#\calO_v\bmod \pi^{n-1}}.$$
Now, we can divide this limit according to the image of $x'$ in $\calO_v/\pi\calO_v$. We can apply our previous result for $x'\not\equiv0\bmod \pi$ and argue inductively for those $x'\equiv0\bmod \pi$. Hence, we obtain that the above limit is $1/2$.
\end{proof}

\begin{remark}
As one can see from the proof, when $r\neq0\bmod \pi$, the quantity in question is in fact equal to $1/2$ for $n$ sufficiently large enough.
\end{remark}

The following limit follows immediately from Lemma \ref{lem:normdist}.
$$\lim_{n\to\infty} \frac{\#\{x\in\calO_v^\times\mid x\in \N(\calO_w)\}\bmod \pi^n}{\#\calO_v^\times \bmod \pi^n}=\frac{1}{2}.$$

\begin{remark}
The above lemma fails when $v$ does not lie over $2$. Indeed whether a unit $x$ is a norm or not can be determined by looking modulo $\pi$.
\end{remark}

Let $\calO_v^*=\calO_v\setminus\{0\}$ be the nonzero elements (not to be confused with $\calO_v^\times$, the nonzero units). Define the sets
$$A=\{x\in\calO_v\mid x(x-1)\in \N(\calO_w^*)\},\quad B=\{x\in\calO_v\mid x-\lambda\in \N(\calO_w^*)\}.$$
Since $x\mapsto x(x-1)$ and $x\mapsto x-\lambda$ are continuous endomorphisms on $\calO_v$, both $A,B$, being inverse images of the open and closed subset $\N(\calO_w^*)$, open and closed inside $\calO_v\setminus\{0,1\}$ and $\calO_v\setminus\{\lambda\}$ respectively. Our first goal is to establish the following.

\begin{prop}\label{prop:limAB}
$$\lim_{n\to\infty}  \frac{\# A \bmod \pi^n }{\#\calO_v \bmod \pi^n} =\lim_{n\to\infty}  \frac{\# B \bmod \pi^n }{\#\calO_v \bmod \pi^n} =\frac{1}{2}$$
\end{prop}

\begin{proof} Since $\{x-\lambda\mid x\in\calO_v\}=\calO_v$, the limit for $B$ follows immediately in view of Lemma \ref{lem:normdist}. To prove the limit for $A$, we first divide $\calO_v$ in the following way
$$\calO_v^{(0)}=\{x\in \calO_v\mid x\equiv0\bmod \pi\},\calO_v^{(1)}=\{x\in \calO_v\mid x\equiv1\bmod \pi\},\calO_v'=\calO_v\setminus (\calO_v^{(0)}\cup \calO_v^{(1)}).$$
We divide the set $A$ in the analogous way
$$A^{(0)}=\{x\in A\mid x\equiv0\bmod \pi\},A^{(1)}=\{x\in A\mid x\equiv1\bmod \pi\},A'=A\setminus (A^{(0)}\cup A^{(1)}).$$
It suffices to show each of the following limits,
$$\lim_{n\to\infty}\frac{\# A^{(0)}\bmod \pi^n}{\#\calO_v^{(0)}\bmod \pi^n}=\lim_{n\to\infty}\frac{\# A^{(1)}\bmod \pi^n}{\#\calO_v^{(1)}\bmod \pi^n}=\lim_{n\to\infty}\frac{\# A'\bmod \pi^n}{\#\calO_v'\bmod \pi^n}=\frac{1}{2}.$$
Define the map $f\colon k_v^\times \to k_v\setminus\{1\}$ given by $f(x)=1-1/x$. Observe that $f$ is a bijection. Moreover, for any $x\in\calO_v^*\setminus \{1\}$, $x(x-1)/f(x)=x^2$. Hence, $x(x-1)\in \N(\calO_w^*)$ if and only if $f(x)\in \N(L_w^\times)$. In particular, this means for $x\in A$ if and only if $f(x)\in \N(L^\times_w)$.

Let $n$ be a positive integer and $x\in \calO_v^\times$. Then
$$f(x+\pi^n\calO_v)=1-1/x+\pi^{n}\calO_v=y+\pi^{n}\calO_v$$
where $y=1-1/x\in \calO_v$. Observe that $y\not\equiv1\bmod \pi$. Hence, $f$ induces a bijection between the following two sets,
$$\{x+\pi^n\calO_v\mid x\in \calO_v^\times\}\xleftrightarrow{f} \{y+\pi^n\calO_v\mid y\in\calO_v,y\not\equiv1\bmod \pi\}.$$
We may further decompose into the following bijections
$$\{x+\pi^n\calO_v\mid x\in \calO_v,x\not\equiv0,1\bmod \pi\}\xleftrightarrow{f}\{y+\pi^n\calO_v\mid y\in\calO_v,y\not\equiv0,1\bmod \pi\},$$
$$\{x+\pi^n\calO_v\mid x\in \calO_v,x\equiv1\bmod \pi\}\xleftrightarrow{f} \{y+\pi^n\calO_v\mid y\in\calO_v,y\equiv 0\bmod \pi\}.$$
Written another way, $f$ induces bijections
$$\calO_v'\bmod \pi^n\xleftrightarrow{f} \calO_v'\bmod \pi^n,$$
$$\calO_v^{(1)}\bmod \pi^n\xleftrightarrow{f} \calO_v^{(0)}\bmod \pi^n.$$
Moreover, under this bijection, $A\bmod \pi^n$ maps to
$$A'\bmod \pi^n\xleftrightarrow{f} \{y\in\calO_v\mid y\in \N(\calO_w), y\not\equiv 0,1\bmod \pi\}\bmod \pi^n,$$
$$A^{(1)}\bmod \pi^n\xleftrightarrow{f} \{y\in\calO_v\mid y\in \N(\calO_w), y\equiv 0\bmod \pi\}\bmod \pi^n.$$
Therefore,
$$\lim_{n\to\infty}\frac{\# A'\bmod \pi^n}{\#\calO_v'\bmod \pi^n}=\lim_{n\to\infty}\frac{\# \{y\in\calO_v\mid y\in \N(\calO_w), y\not\equiv 0,1\bmod \pi\}\bmod \pi^n}{\#\calO_v'\bmod \pi^n}=\frac{1}{2},$$
$$\lim_{n\to\infty}\frac{\# A^{(1)}\bmod \pi^n}{\#\calO_v^{(1)}\bmod \pi^n}=\lim_{n\to\infty}\frac{\# \{y\in\calO_v\mid y\in \N(\calO_w), y\equiv 0\bmod \pi\}\bmod \pi^n}{\#\calO_v^{(0)}\bmod \pi^n}=\frac{1}{2}$$
by Lemma \ref{lem:normdist}. It remains to prove the limit for $A^{(0)}$. For this, consider the map $g\colon \calO_v\to \calO_v$ given by $g(x)=1-x$. This is clearly a bijection and sends $\calO_v^{(0)}$ to $\calO_v^{(1)}$. Moreover, since $x(x-1)=(g(x))(g(x)-1)$, $g$ sends $A^{(0)}$ bijectively to $A^{(1)}$. Hence,
\[\lim_{n\to\infty}\frac{\# A^{(0)}\bmod \pi^n}{\#\calO_v^{(0)}\bmod \pi^n}=\lim_{n\to\infty}\frac{\# A^{(1)}\bmod \pi^n}{\#\calO_v^{(1)}\bmod \pi^n}=\frac{1}{2}.\qedhere\]
\end{proof}

\subsubsection{Applying the equidistribution results}

Finally, we return to the proof of Proposition \ref{prop:linearevenram}.

\begin{proof}[Proof of Proposition \ref{prop:linearevenram}]
We first consider the case when $(a,c)=0$. This means in particular that the fiber over $\infty\in\P^1$ has a point $Q\in X_{\infty}(k_v)$ and $\inv_v\calA(Q)=\inv_v\calB(Q)=0$. So it suffices to find another point with invariant $1/2$.

\begin{Lemma}\label{lem:AmeetB}
$$\lim_{n\to\infty} \frac{\#\calO_v\setminus(A\cup B)\bmod \pi^n}{\#\calO_v \bmod \pi^n}>0.$$
\end{Lemma}

\begin{proof}
Since $X(k_v)\neq\emptyset$ by assumption, there exists $x\in \calO_v$ with $x\neq0,1,\lambda$ such that $x(x-1)(x-\lambda)\in \N(\calO_w)$. This means either $x\in \calO_v\setminus (A\cup B)$ or $x\in A\cap B$. In the former case, the lemma then follows since $A, B$ are open in $\calO_v$. In the latter case $\lim_{n\to\infty} (\#A\cap B)/(\#\calO_v \bmod \pi^n)>0$, and so combining with Proposition \ref{prop:limAB} gives the desired result. 
\end{proof}

Let $x\in \calO_v\setminus (A\cup B)$ and $x\neq0,1,\lambda$. This means $x(x-1)(x-\lambda)\in \N(L_w)$ but $x(x-1)\notin \N(L_w)$. Let $Q\in X(k_v)$ be a point with $x$-coordinate is $x$. Then
$$\inv_v\calA(Q)=(a,x(x-1))=1/2.$$

For the case $(a,c)\neq0$, we have the following

\begin{Lemma}\label{lem:AminusB}
$$\lim_{n\to\infty} \frac{\#(A\setminus B)\bmod \pi^n}{\#\calO_v \bmod \pi^n}>0,\quad \lim_{n\to\infty} \frac{\#(B\setminus A)\bmod \pi^n}{\#\calO_v \bmod \pi^n}>0.$$
\end{Lemma}

\begin{proof}
By Proposition \ref{prop:limAB}, it suffices to show at least one of the limit is positive since that will imply the other is positive as well. Since $X(k_v)\neq\emptyset$ by assumption, there exists $x\in \calO_v$ with $x\neq0,1,\lambda$ such that one of the following two cases happen

{\it Case 1, $x(x-1)\in \N(\calO_w)$ and $x-\lambda\notin \N(\calO_w)$}  OR

{\it Case 2, $x(x-1)\notin \N(\calO_w)$ and $x-\lambda\in \N(\calO_w)$}.

Either case, at least one of $A\setminus B$ or $B\setminus A$ is nonempty, and the limit must also be positive since both are open in $\calO_v$.
\end{proof}

To finish the proof, we choose $x_1\in A\setminus B$ and $x_2\in B\setminus A$. Then there exists $Q_1,Q_2\in X(k_v)$ with $x$ coordinate corresponding to $x_1,x_2$ respectively. It follows that
\[\inv_v\calA(Q_1)=0,\quad \inv_v\calA(Q_2)=1/2.\qedhere\]
\end{proof}

\subsection{Proof of Theorem \ref{thm:SurjLinear}}

\begin{proof}[Proof of Theorem \ref{thm:SurjLinear}]
Let $X$ be a Ch\^atelet surface where $P(x)$ has the form \eqref{eqn:linearfactors}. First assume that for every nonarchimedian place $v$, either $X$ has good reduction or $\sqrt{a}\in k_v$. In the first case, $k_v(\sqrt{a})/k_v$ must be unramified, so that any generator $\alpha$ listed in \S\ref{sec:Brauergroup} is in the kernel of $\Br X\to \Br X_{k_v^{\textup{nr}}}$ where $k_v^{\textup{nr}}$ is the maximal unramified extension of $k_v$. Then by \cite{CTS13}*{Lemma 2.2} the evaluation map $\ev_\alpha\colon X(k_v)\to \Br k_v[2]$ must be constant for any $\alpha\in \Br X$. In the latter case, $\ev_\alpha$ is also constant since the Brauer classes listed in \S3 are trivial over $k_v$. Moreover, if $a>0$ or $k$ does not have a real embedding, then for any archimedian place $v$, $\sqrt{a}\in k_v$, so the evaluation map is constant again. Since $X(k)$ is clearly nonempty (take any root of $P(x)$), it follows $X$ satisfies weak approximation.

Conversely, assume either $X$ has a place $v$ of bad reduction with $\sqrt{a}\notin k_v$ or $v$ is a real place and $a<0$. To show failure of weak approximation, it suffices to show that there is a Brauer--Manin obstruction given by the surjectivity of the evaluation map $\ev_\calA\colon X(k_v)\to \Br k_v[2]$.

If $a<0$, then the evaluation map is surjective at $v$ since taking $x$ such that exactly two of $x,x-1,x-\lambda$ is negative gives rise to a real point $Q$ where either $\ev_\calA(Q)$ or $\ev_\calB(Q)$ is nontrivial. On the other hand, taking $x$ so that all $x,x-1,x-\lambda$ are positive gives rise to a real point $Q$ such that $\ev_\calA(Q)=\ev_\calB(Q)=0$.

Now assume $v$ is a place of bad reduction. If $k(\sqrt{a})/k$ is unramified at $v$, then one of $v(c),v(\lambda),v(\lambda-1)$ must be nonzero. Then Proposition \ref{prop:linearunramA} implies the result. If $k(\sqrt{a})/k$ is ramified then Proposition \ref{prop:linearoddram} for odd $v$ or Proposition \ref{prop:linearevenram} for even $v$ gives the desired result.
\end{proof}

\begin{proof}[Proof of Theorem \ref{thm:BrauerVanishIntro}(1)]
Let $L$ be the splitting field of $P(x)$ as stated in the theorem. If $v$ is a place of bad reduction as given in the hypothesis, then there exists a place $w$ of $L$ lying over $v$ such that $X_L$ has bad reduction at $w$ and $a\notin L_w^{\times2}$. Theorem \ref{thm:SurjLinear} then implies that $X_L$ fails weak approximation.
\end{proof}

\begin{proof}[Proof of Theorem \ref{thm:BrauerVanishIntro}(2)(a)]
Let $K/k$ be a finite extension. By our assumption, either $P(x)$ has an irreducible factor of degree $3$ or splits completely over $K$. In the first case, there is no obstruction to weak approximation by Proposition \ref{Brauergroupgenerators}. In the latter case, if $\sqrt{a}\in K$, then $X_K$ is rational so weak approximation holds as well. Hence, assume $\sqrt{a}\notin K$. Then Theorem \ref{thm:BrauerVanishIntro} implies that $X_K$ satisfies weak approximation.
\end{proof}

\begin{example}\label{ex:perp}
Let us give an example of a surface satisfying the conditions of Theorem \ref{thm:BrauerVanishIntro}(2)(a). Let $\omega$ be a primitive cube root of unity and let $k=\Q(\omega,\sqrt{97})$. Then $k$ has class number $2$ and so the Hilbert class field $L=k(\sqrt{a})$ is an everywhere unramified quadratic extension. Let $X$ be the Ch\^atelet surface
\[
y^2-az^2=x(x-1)(x+\omega).
\]
Observe that $X$ has everywhere good reduction and satisfies the conditions given in Theorem \ref{thm:BrauerVanishIntro}(2)(a). Hence $X$ satisfies weak approximation over all finite field extensions.
\end{example}

\section{Weak approximation in the quadratic case}\label{sec:quadratic}

In this section, we consider the case when $P(x)$ factors as
$$y^2-az^2=cP_1(x)P_2(x)$$
where $P_1,P_2$ are irreducible monic quadratic polynomials. By \S2, the Brauer group modulo $\Br_0 X$ is generated by the quaternion algebra
$$\calC=(a,P_1(x))=(a,cP_2(x)).$$
If the above Brauer class is constant (meaning it comes from $\Br k$), then $X$ satisfies weak approximation. Hence, for the rest of this section, we assume that the class above is nonconstant. This is equivalent to the fact that $\sqrt{a}$ is not in the splitting field of $P_1(x)$ or $P_2(x)$. Moreover, after a change of coordinates, we assume the coefficients of $P_1(x),P_2(x)$ are in $\calO_k$.

Let $v$ be a nonarchimedian place of $k$ with odd residue characteristic and $\pi\in\calO_k$ a uniformizer.

\begin{Lemma}\label{lem:squares}
Let $R(x)\in \F_v[x]$ be a monic irreducible quadratic polynomial. Then for exactly $(q-1)/2$ many of the values $x\in\F_v$, $R(x)$ is a square in $\F_v$.
\end{Lemma}

\begin{proof}
It suffices to show that $R(x)=y^2$ has $q-1$ solutions in $(x,y)\in\F_v^2$  (since $R(x)$ is irreducible, $y$ is never $0$). We may homogenize to define a smooth conic in $\P^2_{\F_v}$. Since this conic has two points at infinity, it is isomorphic to $\P^1_{\F_v}$ and thus has $q+1$ points. Removing the two points at infinity gives $q-1$ solutions to the original equation.
\end{proof}

\begin{prop}\label{prop:quadsurjective}
Assume that $v(a)=1$. If $P_1(x),P_2(x)$ are irreducible modulo $\pi$, then there is an obstruction to weak approximation.
\end{prop}

\begin{proof}
Write
$$P_i(x)=x^2+d_ix+r_i$$
where $d_i,r_i\in\calO_v$. Since $P_i(x)$ is irreducible modulo $\pi$, we must have $r_i\in\calO_v^\times$. Suppose $X$ has a $k_v$ point on a smooth fiber $x=x_0\in k_v$. If $v(x_0)<0$, then the fiber over $\infty\in\P^1$ also has a $k_v$ point. Applying the automorphism $x\mapsto 1/x$ on $\P^1$, we may rewrite the equation for $X$ as
$$y^2-az^2=cr_1r_2(x^2+d_1x/r_1+1/r_1)(x^2+d_2x/r_2+1/r_2)$$
which has a $k_v$ point over the smooth fiber $x=0$. Hence, we may reduce to the case where there exists a point $Q_0=[y_0,z_0,x_0]\in X(k_v)$ on a smooth fiber where $v(x_0)\geq0$.

It suffices to find a point $Q_1\in X(k_v)$ such that $\inv_v \calA(Q_0)\neq \inv_v \calA(Q_1)$. Let $\alpha=x_0\bmod \pi\in \F_v$. By Lemma \ref{lem:squares}, there must be another $\beta\in \P^1(\F_v)$ such that
\begin{enumerate}
\item $P_1(\alpha)$ is a square if and only if $P_1(\beta)$ is a nonsquare, and
\item $P_2(\alpha)$ is a square if and only if $P_2(\beta)$ is a nonsquare.
\end{enumerate}
Here we take the convention that $P_1(\infty)=P_2(\infty)=1$ are squares. This is only needed if both $P_1(\alpha),P_2(\alpha)$ are nonsquares. Then $P_1(\beta)P_2(\beta)$ is nonzero and in the same square class as $P_1(\alpha)P_2(\alpha)$. Therefore, we may use Hensel's lemma to lift to a point $Q_1=(y_1,z_1,x_1)\in X(k_v)$ where $x_1\equiv \beta\bmod \pi$. But then
$$\inv_v \calC(Q_0)=(a,x_0)\neq (a,x_1)=\inv_v \calC(Q_1).$$
The inequality of the two Hilbert symbols is due to (1) and (2) above.
\end{proof}

If the hypothesis of Proposition \ref{prop:quadsurjective} does not hold, the existence of an obstruction to weak approximation is much more intrinsic to the surface in question. In particular, one cannot expect a uniform result similar to the case when $P(x)$ splits completely or is irreducible. We illustrate this subtlety by considering two Ch\^atelet surfaces whose defining equation differs by one coefficient. 


\begin{example}\label{ex:WAsatisfied}
Let $X/\Q$ be the Ch\^atelet surface given by $$y^2-17z^2=3(x^2-7)(17x^2-43).$$ By Proposition \ref{Brauergroupgenerators}, $\Br X/\Br_0 X$ is generated by $\mathcal{C}=(17,x^2-7)=(17,3(17x^2-43))$. We show that $X$ satisfies weak approximation.

We begin by showing $X(\Q_p)\neq \emptyset$ for all primes $p$. First, observe that for $p\neq 3,17$ we have $X_{\infty}(\Q_p)\neq \emptyset$ because $X_{\infty}$ is the conic $y^2-17z^2=51w^2$, which has $\Q_p$ points for $p\neq 3,17$. Next, we observe that the fiber $X_1$ is the conic $ y^2-17z^2=468w^2$ which has $\Q_3$ and $\Q_{17}$-points since $468$ is a square in $\Q_3$ and $\Q_{17}$.

We show that the map $\ev_{\mathcal{C}} \colon X(\Q_p) \rightarrow  \Br \Q_p$ is constant at all primes $p$. Since the evaluation map is constant at all primes of good reduction \cite{CTS13}*{Lemma 2.2}, it remains to check $\ev_{\mathcal{C}}\left(X(\Q_p)\right)$ for the primes $p=2,3,7,17$ and $43$. Let $(x_0,y_0,z_0)\in X(\Q_p)$ and $P(x)=3(x^2-7)(17x^2-43)$. 

($p=2$) Note that $\inv_2( \ev_{C}\left(X(\Q_2)\right))=0$ because $17 \in \Q_2^{\times 2}$. 

($p=17$) Assume $v_{17}(x_0)<0$ so then $v_{17}\left(P(x_0)\right)=4v_{17}(x_0)+1$. Furthermore, we can see that $$17^{-(4v_{17}(x_0)+1)}P(x_0) \equiv 3 \pmod {17} $$ which is not a square in $\Q_{17}$. We can now conclude that in this case, $P(x)$ is never a norm from the ramified extension $\Q_{17}(\sqrt{17})$. Hence we must have $v_{17}(x_0)\geq 0$. For these values of $x_0$, $3(17x_0^2-43) \equiv 7 \pmod {17}$ is not a square so $\inv_{17}(\ev_{\calC}(X(\Q_{17})))=1/2$.

($p=7$) Since $3(17x^2-43)$ is irreducible over $\Q_7$, $v_7(3(17x_0^2-43))$ must be even. Hence $\inv_7( \ev_{\calC}\left(X(\Q_7)\right))=0$.

($p=43$) Since $x^2-7$ is irreducible over $\Q_{43}$, we have $v_{43}(x_0^2-7)$ must be even. Hence $\inv_{43}( \ev_{\calC}\left(X(\Q_{43})\right))=0$.

($p=3$) Since $17x^2-43$ is irreducible over $\Q_3$, $v_3(17x_0^2-43)$ must be even. Hence $v_3(3(17x_0^2-43))$ is odd and so $\inv_3(\ev_{\calC}(X(\Q_3)))=1/2$.

Combining the above calculations, we obtain that the sum of invariants is always $0$, which means weak approximation holds.
 \end{example}

After a minor adjustments to the surface of Example \ref{ex:WAsatisfied}, we obtain another Ch\^atelet surface which fails weak approximation.

\begin{example}\label{ex:WAfailed}
Let $X/\Q$ be the Ch\^atelet surface given by $$y^2-17z^2=3(x^2-7)(17x^2-7\cdot43).$$ We first show that $X(\A_{\Q})\neq \emptyset$. In a similar fashion to Example \ref{ex:WAsatisfied} we have $X(\Q_p)\neq \emptyset$ for all $p\neq 3,17$ hence only non-emptiness of $X(\Q_3)$ and $X(\Q_{17})$ must be checked directly. One can compute the fibers $X_1$ and $X_3$ and see that the conic $X_1$ has $\Q_3$-points and the conic $X_3$ has $\Q_{17}$-points.

We claim that $X$ fails weak approximation and to prove this, it suffices to verify that the map $\ev_{\mathcal{C}} \colon X(\Q_7) \rightarrow \Br \Q_7[2]$ is surjective. To see this, first note that $X_\infty(\Q_7)\neq\emptyset$ and $\ev_{\calC}$ is trivial over such a point. On the other hand, as $P(0)=3\cdot7^2\cdot 43$ has even valuation, there exists $Q\in X_0(\Q_7)$ lying over $0\in\P^1(\Q_7)$ where
$$\ev_{\calC}(Q)=(17,-7)\neq0.$$
This proves our claim.
\end{example}

\end{document}